\documentclass[a4paper,11pt]{article}
\usepackage{graphicx}
\usepackage{amsmath,amssymb,amsthm}
\usepackage{mhsymb}
\usepackage{mathtools}
\usepackage{hyperref}
\usepackage{breakurl}
\usepackage{mhequ}
\usepackage{mathrsfs}

\usepackage{microtype}
\usepackage{wasysym}
\usepackage{centernot}
\usepackage{enumitem}

\usepackage{geometry}

\theoremstyle{plain}
\newtheorem{theorem}{Theorem}[section]
\newtheorem{lemma}[theorem]{Lemma}
\newtheorem{proposition}[theorem]{Proposition}
\newtheorem{corollary}[theorem]{Corollary}
\newtheorem{assumption}[theorem]{Assumption}

\theoremstyle{definition}
\newtheorem{definition}[theorem]{Definition}
\newtheorem{remark}[theorem]{Remark}
\newtheorem{example}[theorem]{Example}

\newcommand{\eqlaw}{\stackrel{\mbox{\tiny law}}{=}}

\newcommand{\T}{\mathbb{T}}
\def\R{\mathbb{R}}
\def\E{\mathbb{E}}
\def\P{\mathbb{P}}
\def\Z{\mathbb{Z}}
\def\C{\mathbb{C}}

\newcommand{\mfK}{\mathfrak{K}}

\newcommand{\mfg}{\mathfrak{g}}

\newcommand{\mcS}{\mathcal{S}}

\newcommand{\mcI}{\mathcal{I}}

\newcommand{\mcX}{\mathcal{X}}

\newcommand{\F}{\mathbf{F}}

\newcommand{\mbi}{\mathbf{i}}

\newcommand{\mrd}{\mathop{}\!\mathrm{d}}
\newcommand{\init}{\mathcal{I}}
\def\id{\mathrm{id}}

\def\emptyset{{\centernot\Circle}}

\numberwithin{equation}{section}

\def\dash{\leavevmode\unskip\kern0.18em--\penalty\exhyphenpenalty\kern0.18em}
\def\slash{\leavevmode\unskip\kern0.15em/\penalty\exhyphenpenalty\kern0.15em}

\let\OLDthebibliography\thebibliography
\renewcommand\thebibliography[1]{
  \OLDthebibliography{#1}
  \setlength{\parskip}{0pt}
  \setlength{\itemsep}{2mm}
}

\title{Norm inflation for a non-linear heat equation with Gaussian initial conditions}

\author{
Ilya Chevyrev\thanks{School of Mathematics, The University of Edinburgh, Edinburgh EH9 3FD, United Kingdom.
\href{mailto:ichevyrev@gmail.com}{\tt ichevyrev@gmail.com}}
}

\date{}

\begin{document}

\maketitle

\begin{abstract}
We consider a non-linear heat equation $\partial_t u = \Delta u + B(u,Du)+P(u)$ posed on the $d$-dimensional torus, where $P$ is a polynomial of degree at most $3$ and $B$ is a bilinear map that is not a total derivative. We show that, if the initial condition $u_0$ is taken from a sequence of smooth Gaussian fields with a specified covariance, then $u$ exhibits norm inflation with high probability. A consequence of this result is that there exists no Banach space of distributions which carries the Gaussian free field on the 3D torus and to which the DeTurck--Yang--Mills heat flow extends continuously, which complements recent well-posedness results of Cao--Chatterjee and the author with Chandra--Hairer--Shen. Another consequence is that the (deterministic) non-linear heat equation exhibits norm inflation, and is thus locally ill-posed, at every point in the Besov space $B^{-1/2}_{\infty,\infty}$; the space $B^{-1/2}_{\infty,\infty}$ is an endpoint since the equation is locally well-posed for $B^{\eta}_{\infty,\infty}$ for every $\eta>-\frac12$.
\end{abstract}



\section{Introduction}\label{sec:intro}

We consider the initial value problem for a non-linear heat equation of the form
\begin{equ}\label{eq:PDE}
\begin{cases}
\partial_t u = \Delta u + B(u, Du) + P(u) \quad &\text{ on } [0,T]\times \T^d
\\
u(0,\cdot) = u_0(\cdot) \in C^\infty(\T^d,E)\;,
\end{cases}
\end{equ}
where $\T^d = \R^d/2\pi\Z^d$ is the $d$-dimensional torus,
$E$ is a vector space (over $\R$) of dimension $2\leq \dim(E) < \infty$, $B\colon E\times E^d\to E$ is a bilinear map, $Du = (\partial_1 u,\ldots, \partial_d u)$, and $P \colon E\to E$ is a polynomial of degree at most $3$.
In what follows, we assume $B$ is \textit{not a total derivative}, i.e., 
if we write $B$ for $y=(y_1,\ldots, y_d)\in E^d$ as
\begin{equ}
B(\cdot,y) = B_1(\cdot,y_1)+\ldots + B_d(\cdot,y_d)\;,
\end{equ}
where $B_i\colon E\times E\to E$ is bilinear, then we assume that $B_i$ is not symmetric for some $1\leq i\leq d$.
We note that $B_i$ is symmetric if and only if there exists a bilinear map $\tilde B_i\colon E\times E\to E$ such that $\partial_i \tilde B_i(u,u) = B_i(u,\partial_i u)$ for all smooth $u$.

It is easy to show that, for $\eta>-\frac12$, a solution $u$ to~\eqref{eq:PDE} exists for $T>0$ sufficiently small depending (polynomially) on $|u_0|_{\CC^\eta}$, where $\CC^\eta= B^{\eta}_{\infty,\infty}$ is the H\"older--Besov space of regularity $\eta$ (see Section~\ref{subsec:Besov} for a definition of $B^\alpha_{p,q}$).
Furthermore, the map $\CC^\eta \ni u_0\mapsto u$ is locally Lipschitz once the target space is equipped with a suitable norm and $T$ is taken sufficiently small on each ball in $\CC^\eta$.
We give a proof of these facts in Appendix \ref{app:large_eta}.
(If each $B_i$ were symmetric, these facts would hold for $\eta>-1$.)

The goal of this article is to prove a corresponding local ill-posedness result for a family of function spaces that carry sufficiently irregular Gaussian fields.
Leaving precise definitions for later, the main result of this article can be stated as follows.

\begin{theorem}\label{thm:main_informal}
Suppose that $X$ is an $E$-valued Gaussian Fourier series (GFS) on $\T^d$ whose Fourier truncations $\{X^N\}_{N\geq 1}$ satisfy Assumption~\ref{assump:GFS} below.
Then there exists another $E$-valued GFS $Y$, defined on a larger probability space, such that $Y\eqlaw X$ and, for every $\delta>0$,
\begin{equ}
\lim_{N\to\infty}\P\Big[\sup_{t\in [0,\delta]} \Big|\int_{\T^d}u_t(x)\mrd x\Big| > \delta^{-1}\Big] = 1\;,
\end{equ}
where $u$ is the solution to~\eqref{eq:PDE} with $u_0=X^N+Y^N$
and if $u$ blows up at or before time $\delta$, then we treat
$\sup_{t\in [0,\delta]} |\int_{\T^d}u_t(x)\mrd x| = \infty$.
\end{theorem}

We formulate a more general and precise statement in Theorem~\ref{thm:main} below.
See the end of Section~\ref{sec:prelim} for a description of the proof and of the structure of the article.
The definition of an $E$-valued GFS is given in Section~\ref{subsec:GFS}.

\begin{remark}
The Gaussian free field $X$ on $\T^d$ with $\E|X_k|^2 = |k|^{-2}$ for $k\neq 0$ (see e.g.~\cite{Sheffield_GFF}) is a GFS (for any $d\geq 1$) and $X$ satisfies
Assumption~\ref{assump:GFS} for $d=3$ (but not other values of $d$).
\end{remark}

There is considerable interest in studying partial differential equations (PDEs) with random initial conditions in connection to local and global well-posedness;
for dispersive PDEs, see for instance~\cite{Bourgain94, Burq_Tzvetkov_08,Oh_Pocovnicu_16,Pocovnicu17} and the review article~\cite{BOP19}; for recent work on parabolic PDEs, see~\cite{Sourav_flow, CCHS22, Hairer_Le_Rosati_22}.
Furthermore, there have been many developments in the past decade in the study of \textit{singular stochastic} PDEs with both parabolic~\cite{Hairer14,GIP15,BCCH21} and dispersive~\cite{GKO18,GKOT21,Tolomeo21} dynamics, for which it is often important to understand the effect of irregularities of (random) initial condition.

In~\cite{Sourav_flow,CCHS22}, it is shown that~\eqref{eq:PDE} is locally \textit{well-posed}\footnote{Strictly speaking, only the DeTurck--Yang--Mills heat flow, which is a special form of~\eqref{eq:PDE}, was considered in~\cite{Sourav_flow,CCHS22} but the methods therein apply to the general form of~\eqref{eq:PDE}; furthermore, only mollifier approximations were considered in~\cite{CCHS22},
but the methods apply to Fourier truncations.}
for the GFF $X$ on $\T^3$, in the sense that, if $u_0 = X^N$, then $u$ converges in probability as $N\to\infty$ to a continuous process with values in $\CC^{\eta}$ for $\eta<-1/2$, at least over short random time intervals; we discuss these works in more detail at the end of this section.
Theorem~\ref{thm:main_informal} therefore shows the importance of taking precisely the GFF $X$ in these works, rather than another Gaussian field which has the same regularity as $X$ when measured in any normed space.
We note that $Y$ in Theorem~\ref{thm:main_informal} can therefore clearly \textit{not} be taken independent of $X$.

We now state a corollary of Theorem~\ref{thm:main_informal} and of the construction of $Y$ which elaborates on this last point and is of analytic interest.
Consider a Banach space $(\init,\|\cdot\|)$ of distributions with (continuous) inclusions $C^\infty\subset \init \subset \mcS'(\T^d,\R)$.
Let
\begin{equ}
(\hat\init,\|\cdot\|)\;,\qquad \hat\init\subset \mcS'(\T^d,E)\;,
\end{equ}
denote the Banach space of all $E$-valued distributions of the form $\sum_{a\in A} x^aT^a$, where each $x^a\in \init$ and $\{T^a\}_{a\in A}$ is some fixed basis of $E$, with norm $\|\sum_{a\in A} x^aT^a\| \eqdef \sum_a \|x^a\|$.
We say that there is norm inflation for~\eqref{eq:PDE} at $x\in \hat\init$ if, for every $\delta>0$, there exists a solution $u$ to~\eqref{eq:PDE} such that
\begin{equ}\label{eq:norm_inflation}
\|x-u_0\|<\delta\qquad \text{ and } \qquad \Big|\int_{\T^d} u_t(x)\mrd x\Big|> \delta^{-1}\;.
\end{equ}
This notion of norm inflation is slightly stronger than the usual one introduced by Christ--Colliander--Tao~\cite{CCT03} in their study of the wave and Schr\"odinger equations,
for which only $x=0$ and $\|u_t\| > \delta^{-1}$ are considered in~\eqref{eq:norm_inflation}.
Norm inflation (in the sense that $\|u_t\| > \delta^{-1}$) at arbitrary points $x\in\hat\mcI$ was shown by Xia~\cite{Xia21} for the non-linear wave equation below critical scaling.
Oh~\cite{Oh17} showed that there is norm inflation for the cubic non-linear Schr\"odinger equation at every point in negative Sobolev spaces at and below the critical scaling.

Norm inflation captures a strong form of ill-posedness of~\eqref{eq:PDE}:~\eqref{eq:norm_inflation} in particular shows that the solution map $C^\infty\ni u_0 \mapsto u\in C([0,T],\mcS')$ is discontinuous at $x$ in $\hat\init$, even locally in time.
See~\cite{Iwabuchi_Ogawa_15,Carles_Kappeler_17,Oh_Wang_18, Kishimoto19} and~\cite{BP08,Cheskidov_Dai_14, Molinet_Tayachi_15} where norm inflation of various types is shown for dispersive and parabolic PDEs respectively.

Consider now a family of non-negative numbers $\sigma^2 = \{\sigma^2(k)\}_{k\in\Z^d}$.
We say that $\init$ \textit{carries the GFS with variances $\sigma^2$} if
\begin{equ}\label{eq:GFS_conv}
\lim_{N\to\infty}\|X-X^N\| = 0 \quad \text{in probability}
\end{equ}
whenever $\{X_k\}_{k\in\Z^d}$ is a family of (complex) Gaussian random variables such that $\E|X_k|^2=\sigma^2(k)$ and $X^N \eqdef \sum_{|k|\leq N} X_k e^{\mbi\scal{k,\cdot}}$
is a real GFS for every $N\geq 1$ (see Definition~\ref{def:GFS}), and
where $X$ is a real GFS.
\begin{example}
For $d=1$ and $\eta<-\frac12$, $\init\eqdef\CC^\eta$ carries the GFS with variances $\sigma^2\equiv 1$, which corresponds to white noise on $\T^1$ and which satisfies Assumption~\ref{assump:GFS}.
\end{example}

\begin{corollary}\label{cor:norm_inflation}
Suppose $\init$ carries the GFS with variances $\sigma^2= \{\sigma^2(k)\}_{k\in\Z^d}$ satisfying the bounds in Assumption~\ref{assump:GFS}.
Suppose that the law of $X$ in~\eqref{eq:GFS_conv} has full support in $\init$.
Then there is norm inflation for~\eqref{eq:PDE} at every point in $\hat\init$.
\end{corollary}
The proof of Corollary~\ref{cor:norm_inflation} is given in Section~\ref{subsec:main_results}.
We note that a simple criterion for the law of $X$ to have full support in $\init$ is that every $\sigma^2(k)>0$ and that the smooth functions are dense in $\init$ (see Remark~\ref{rem:full_support}).

Note that the Besov space $\CC^{-1/2}=B^{-1/2}_{\infty,\infty}$ is an `endpoint' case since~\eqref{eq:PDE} is well-posed on $\CC^{\eta}$ for every $\eta>-\frac12$.
In Proposition~\ref{prop:Besov_reg}\ref{pt:-1} below, we show that $\CC^{-1/2}$
carries a GFS with variances $\sigma^2$ satisfying the bounds in Assumption~\ref{assump:GFS}.
Corollary~\ref{cor:norm_inflation}, together with Remark~\ref{rem:full_support} and Proposition~\ref{prop:Besov_reg}\ref{pt:-1}, therefore gives a probabilistic proof of the following result, which appears to be folklore.
\begin{corollary}\label{cor:endpoint}
There is norm inflation for~\eqref{eq:PDE} at every point in $\CC^{-1/2}(\T^d,E)$.
\end{corollary}
We remark that $\CC^{-1}$ is the scaling critical space for~\eqref{eq:PDE}, so our results show ill-posedness above scaling criticality.
Norm inflation of the same type as~\eqref{eq:norm_inflation} in $\CC^{-2/3}$
was shown for the cubic non-linear heat equation (NLH) $\partial_t u= \Delta u \pm u^3$ by the author and Oh--Wang~\cite{COW22}
using a Fourier analytic approach taking its roots in~\cite{Iwabuchi_Ogawa_15, Oh17, Kishimoto19}.
The method of~\cite{COW22} can likely be adapted to yield norm inflation for~\eqref{eq:PDE} in $\CC^{-1/2}$, and we expect that the methods of this article can similarly be adapted to show (probabilistic) norm inflation for the cubic NLH.
We remark, however, that our method appears not to reach
$B^{-1/2}_{\infty,q}$ for $q<\infty$ (see Proposition~\ref{prop:Besov_reg}\ref{pt:<-1}),
while the method in~\cite{COW22} for the cubic NLH does extend to $B^{-2/3}_{\infty,q}$ for $q>3$.

Corollary~\ref{cor:endpoint} should be contrasted with the 3D Navier--Stokes equations (NSE) for which norm inflation (in a slightly weaker sense than~\eqref{eq:norm_inflation}) was shown in the scaling critical space $\CC^{-1}$ by Bourgain--Pavlovi{\'c}~\cite{BP08} and which is locally well-posed in $\CC^\eta$ for $\eta>-1$ (the main difference with~\eqref{eq:PDE} is that the non-linearity in NSE \textit{is} a total derivative).
In fact, norm inflation for NSE has been shown in $B^{-1}_{\infty,q}$ for $q>2$ by Yoneda~\cite{Yoneda10} and for all $q\in [1,\infty]$ by Wang~\cite{Wang15}.
We remark that our analysis (and generality of results) relies on the fact that any space carrying $X$ as in Theorem~\ref{thm:main_informal}, in particular $\CC^{-1/2}$, is scaling \textit{subcritical}.
See also~\cite{Choffrut_Pocovnicu_18} and~\cite{Forlano_Okamoto_20} where norm inflation is established for the fractional non-linear Schr{\"o}dinger (NLS) and non-linear wave (NLW) equations respectively above the critical scaling.

In \cite{SunTz_20_norm_inf} and \cite{Camps_Gassot_23}, norm inflation is shown for the NLW and NLS respectively for all initial data $u_0$ belonging to a dense $G_\delta$ subset of the scaling supercritical Sobolev space and where the approximation $x$ is taken as a mollification of $u_0$.
These works in particular imply a statement similar to Corollary \ref{cor:endpoint} for the NLW and NLS but with a more precise description of the approximating sequence that exhibits norm inflation (cf. \cite{Oh17,Xia21}).

We finish the introduction by describing one of the motivations for this study.
The author and Chandra--Hairer--Shen in~\cite{CCHS20,CCHS22} analysed the stochastic quantisation equations of the Yang--Mills (YM) and YM--Higgs (YMH) theories on $\T^2$ and $\T^3$ respectively (see also the review article~\cite{Chevyrev22_YM});
we also mention that
Bringmann-Cao~\cite{BC23} analysed the YM stochastic quantisation equations on $\T^2$ by means of para-controlled calculus (vs. regularity structures as in \cite{CCHS20}), and that the invariance of the YM measure on $\T^2$ for this dynamic was shown in \cite{ChevyrevShen23}. 
In~\cite{CCHS22}, the authors in particular constructed a candidate state space for the YM(H) measure on $\T^3$.
A related construction was also proposed by Cao--Chatterjee~\cite{Sourav_state,Sourav_flow}.
An ingredient in the construction of~\cite{CCHS22} is a metric space $\init$ of distributions such that the solution map of the
DeTurck--YM(H) heat flow (or of any equation of the form~\eqref{eq:PDE}, see~\cite[Prop.~2.9]{CCHS22}) extends continuously locally in time to $\init$
and such that suitable smooth approximations of the GFF on $\T^3$ converge in $\init$; essentially the same space was identified in~\cite{Sourav_flow}.
The works~\cite{Sourav_flow,CCHS22} thus
provide a local \textit{well-posedness} result for~\eqref{eq:PDE} with the GFF on $\T^3$ as initial condition.

In contrast, our results provide a corresponding \textit{ill-posedness} result. Indeed, the GFF on $\T^3$ satisfies Assumption~\ref{assump:GFS} and the
DeTurck--YM(H) heat flow is an example of equation~\eqref{eq:PDE} covered by our results (see Examples~\ref{ex:YM_flow} and~\ref{ex:YMH_flow}).
Theorem~\ref{thm:main_informal} and Corollary~\ref{cor:norm_inflation} therefore imply
that the metric space $\init$ in~\cite{CCHS22} is not a vector space, and, more importantly, that this situation is unavoidable.
That is, there exists no Banach space
of distributions which carries the GFF on $\T^3$ and admits a continuous extension of the DeTurck--YM(H) heat flow.
Since the 3D YM measure (at least in a regular gauge) is expected to be a perturbation of the standard 3D GFF,
this suggests that every sensible state space for the 3D YM(H) measure is necessarily non-linear (cf.~\cite{Chevyrev19YM,CCHS20} in which natural \textit{linear} state spaces for the 2D YM measure were constructed).

A non-existence result in the same spirit was shown for iterated integrals of Brownian paths by Lyons~\cite{TerryPath} (see also~\cite[Prop.~1.29]{Lyons07} and~\cite[Prop.~1.1]{FrizHairer20});
the construction of the field $Y$ in Theorem~\ref{thm:main_informal} is inspired by a similar one in~\cite{TerryPath}.

\section{Preliminaries and main result}\label{sec:prelim}

\subsection{Fourier series and Besov spaces}
\label{subsec:Besov}

We recall the definition of Besov spaces that we use later.
Thorough references on this topic include~\cite{BookChemin,Triebel};
see also~\cite[Appendix~A]{GIP15} and~\cite[Appendix~A]{MW17_Phi43} for concise summaries.
Let $\chi_{-1},\chi \in C^\infty(\R^d,\R)$ take values in $[0,1]$ such that $\supp \chi_{-1}\subset B_0(4/3)$ and $\supp \chi\subset B_0(8/3)\setminus B_0(3/4)$, where $B_x(r)=\{y\in\R^d\,:\,|x-y|\leq r\}$, and
$
\sum_{\ell=-1}^\infty \chi_\ell = 1
$,
where $\chi_\ell \eqdef \chi(2^{-\ell}\cdot)$ for $\ell\geq 0$.

Let $\mbi=\sqrt{-1}$ and, for $k\in\Z^d$, let $e_k = e^{\mbi \scal{k,\cdot}}\in C^\infty(\T^d)$ denote the orthonormal Fourier basis; here and below, we equip $\T^d = \R^d/2\pi\Z^d$ with the normalised Lebesgue measure of total mass $1$.
For $f\in \mcS'(\T^d)$ denote $f_k = \scal{f,e_k} = \int_{\T^d} f(x) e_{-k}(x)\mrd x$
and, for
$\ell\geq -1$, define
$\Delta_\ell f\in C^\infty(\T^d)$ by
\begin{equ}
\Delta_\ell f = \sum_{k\in\Z^d} \chi_\ell(k) f_k e_k \;.
\end{equ}
For $s\in \R$ and $1\leq p,q \leq \infty$, we define the Besov norm of $f\in C^\infty(\T^d)$ by
\begin{equ}
|f|_{B^\alpha_{p,q}} \eqdef  \Big(\sum_{\ell\geq -1} 2^{\alpha \ell q}|\Delta_\ell f|_{L^p(\T^d)}^q\Big)^{1/q} < \infty\;.
\end{equ}
We let $B^\alpha_{p,q}$ be the space obtained by completing $C^\infty(\T^d)$ with this norm, which one can identify with a subspace of distributions $\mcS'(\T^d)$.
We use the shorthand $\CC^\alpha = B^\alpha_{\infty,\infty}$.

We let $\CP_t = e^{t\Delta}\colon \mcS'(\T^d) \to C^\infty(\T^d)$ for $t>0$ denote the heat flow
(with $\CP_0=\id$ as usual).
In particular, $\CP_t e_k = e^{-|k|^2 t} e_k$ for all $k\in\Z^d$.
We denote by $\Pi_N\colon \mcS'(\T^d) \to C^\infty(\T^d)$ the Fourier truncation operator
\begin{equ}
\Pi_N\xi = \sum_{|k|\leq N} \xi_k e_k\;.
\end{equ}

\subsection{Assumptions on the equation}\label{subsec:assumps_equ}

Consider $E,B,P$ as in Section~\ref{sec:intro}.
Without loss of generality, we will assume $B_1$ is not symmetric
and henceforth make the following assumption.

\begin{assumption}\label{assump:not-deriv}
Fix a basis $\{T^a\}_{a\in A}$ of $E$.
There exist $a\neq b \in A$ such that $B_1(T^a,T^b) \neq B_1(T^b,T^a)$.
\end{assumption}

The next two examples show that our results cover the DeTurck--YM(H) heat flow.

\begin{example}[Yang--Mills heat flow with DeTurck term]\label{ex:YM_flow}
Let $E=\mfg^d$, where $\mfg$ is a non-Abelian finite-dimensional Lie algebra with Lie bracket $[\cdot,\cdot]$.
We write elements of $E$ as $X=\sum_{j=1}^d X^j\mrd x^j$, $X^j\in\mfg$.
The DeTurck--YM heat flow is
\begin{equ}
\partial_t X^j = \Delta X^j + \sum_{i=1}^d[X^i,2\partial_i X^j - \partial_j X^i + [X^i,X^j]]\;,\quad 1\leq j\leq d\;.\label{eq:DYM}\tag{DYM}
\end{equ}
To bring this into the form~\eqref{eq:PDE},
we write elements of $E^d$ as $(\partial_1 Y,\ldots, \partial_d Y)\in E^d$, $\partial_i Y = \sum_j \partial_i Y^j\mrd x^j \in E$, $\partial_i Y^j \in \mfg$.
Define
$B\colon E\times E^d\to E$ by
\begin{equ}
B(X,Y)=\sum_{i,j=1}^d[X^i, 2\partial_i Y^j -\partial_j Y^i] \mrd x^j\;,
\end{equ}
so that the corresponding $B_i\colon E\times E\to E$ is
\begin{equ}
B_i(X,Y) = \sum_{j=1}^d  [X^i,2Y^j - \delta_{i,j} Y^j ]\mrd x^j\;.
\end{equ}
Then $B_i$ is not symmetric for every $1\leq i\leq d$
since $B_i(X,Y) = [X^i, Y^i]\mrd x^i + \sum_{j\neq i} c_j\mrd x^j$ for some $c_j\in\mfg$ and $[X^i, Y^i]$ is \textit{anti-symmetric} and non-zero for some $X,Y\in E$ since $\mfg$ is non-Abelian.
Equation~\eqref{eq:DYM} therefore satisfies Assumption~\ref{assump:not-deriv}.
\end{example}

\begin{example}\label{ex:YMH_flow}
The same example as above shows that the DeTurck--YM--Higgs heat flow (with or without the cube of the Higgs field, see~\cite[Eq.~(1.9) resp.~(2.2)]{CCHS22})
satisfies Assumption~\ref{assump:not-deriv}.
\end{example}

\begin{example}
If $\dim(E)=1$, then Assumption~\ref{assump:not-deriv} is never satisfied.
\end{example}


\subsection{Gaussian Fourier series}\label{subsec:GFS}

\begin{definition}\label{def:GFS}
A \textit{complex Gaussian Fourier series (GFS)} is an $\mcS'(\T^d,\C)$-valued random variable $X$ such that $\{X_k\}_{k\in \Z^d}$ is a family of complex Gaussians
with $\E X_k^2 = 0$ for all $k \neq 0$ and $X_k$ and $X_n$ are independent for all $k,n\in\Z^d$ such that $k\notin \{-n,n\}$.
Here, as before, we denote $X_k\eqdef \scal{X,e_k}$.
A \textit{real GFS} is a complex GFS 
$X$ such that $X_{-k}=\overline X_k$ for all $k\in \Z^d$.
An \textit{$E$-valued GFS} is an $\mcS'(\T^d,E)$-valued random variable such that $X = \sum_{a\in A}X^aT^a$ with $\{X^a\}_{a\in A}$ a family of independent real GFSs.
\end{definition}
The sequence $\{\E|X_k|^2\}_{k\in \Z^d}$ uniquely determines the law of a real GFS $X$.
Conversely, a real GFS can be constructed from any sequence $\{\sigma^2(k)\}_{k\in \Z^d}$ with polynomial growth by taking a set $\mfK\subset \Z^d$ such that $\mfK\cap(-\mfK) = \emptyset$ and $\mfK\cup (-\mfK) = \Z^d\setminus\{0\}$,
defining $\{X_{k}\}_{k\in \mfK}$ as a family of independent complex Gaussians with $\E|X_k|^2=\sigma^2(k)$ and $\E X_k^2=0$, and setting $X_{-k}=\overline X_k$ for all $k \in \mfK$,
and $X_0$ as a real Gaussian with $\E|X_0|^2=\sigma^2(0)$.
\begin{assumption}\label{assump:GFS}
Suppose $\{X^N\}_{N\geq 1}$ is a family of $E$-valued GFSs such that\footnote{For $k\in\Z^d$, we use the shorthand $\log k=\log |k|$ and $k^\gamma=|k|^\gamma$.}
\begin{equ}
C^{-1} k^{-d+1}|\log k|^{-1}|\log\log k|^{-1} \leq \sigma^2_N(k) \eqdef \E|X^N_k|^2 \leq C k^{-d+1}\;,
\end{equ}
for all $k_0 \leq |k|\leq N$, and $\sigma^2_N(k) \leq C$ for $0\leq |k|<k_0$, where $k_0,C>0$
are independent of $N$,
and $\sigma^2(k)=0$ for $|k|>N$. 
\end{assumption}
\begin{example}
Let $X$ be an $E$-valued GFS with $\E|X_k|^2 = k^{-d+1}$ for $k\neq 0$.
Then $X^N\eqdef \Pi_N X$ satisfies Assumption~\ref{assump:GFS}.
\end{example}
\begin{remark}
The upper bound $\sigma^2_N(k) \leq C k^{-d+1}$ in Assumption~\ref{assump:GFS} can be relaxed to $\sigma^2_N(k) \leq C k^\gamma$ for $\gamma\in [-d+1,-d+1+\eps]$ for some sufficiently small $\eps>0$ without changing the statement of Theorem~\ref{thm:main} below.
We restrict to $\gamma=-d+1$ only for simplicity.
\end{remark}

\begin{lemma}\label{lem:Besov_moments}
Let $\eta<-\frac12$ and
suppose that the upper bound on $\sigma^2_N(k)$ in Assumption~\ref{assump:GFS} holds.
Then
$\sup_N \E |X^N|^p_{\CC^\eta}<\infty$ for all $p\in[1,\infty)$.
\end{lemma}

\begin{proof}
This follows from a standard Kolmogorov-type argument (similar to and simpler than the proof of Lemma~\ref{lem:non_zero_modes}).
It also follows directly from the sharper result of Proposition \ref{prop:Besov_reg}\ref{pt:<-1}
combined with the obvious embedding $B^{\alpha}_{p,q} \hookrightarrow B^{\alpha}_{p,\infty}$.
\end{proof}
\begin{remark}
The logarithmic factors in Assumption~\ref{assump:GFS} are considered to allow for endpoint cases; we will see in Proposition~\ref{prop:Besov_reg} that they allow for $X^N \eqdef \Pi_N X$ to converge to a GFS $X$ in $\CC^{-1/2}$, which would not be the case without these factors.
This allows us to show norm inflation for~\eqref{eq:PDE} in $\CC^{-1/2}$ (Corollary~\ref{cor:endpoint}).
If these logarithmic factors are dropped, then
$|\Pi_0 u_T| > c\log\log\log N$
in~\eqref{eq:prob_norm_inflation} can be replaced by $\inf_{t\in [N^{-2+\eps},T]}|\Pi_0 u_t| > c\log N$ for any $\eps>0$.
\end{remark}
\subsection{Main result}\label{subsec:main_results}

\begin{definition}
For a distribution $\xi\in\mcS'(\T^d,\C)$, we define $\CR\xi \in \mcS'(\T^d,\C)$ as the unique distribution such that, for all $k\in\Z^d$,
\begin{equ}
(\CR\xi)_k\eqdef  \scal{\CR\xi,e_k}
=
\begin{cases}
\mbi\xi_k &\quad \textnormal{ if } k_1>0\;,\\
-\mbi\xi_{k} &\quad \textnormal{ if } k_1<0\;,\\
\xi_k &\quad \textnormal{ if } k_1=0\;.
\end{cases}
\end{equ}
\end{definition}
If $\xi$ satisfies the reality condition $\xi_{-n}=\overline{\xi_n}$, then so does $\CR\xi$.
Furthermore, if $X$ is a real GFS, then so is $\CR X$ and $\CR X\eqlaw X$.
Recall now Assumption~\ref{assump:not-deriv}.

\begin{definition}\label{def:Y}
For an $E$-valued GFS $X$, we define another $E$-valued GFS $Y=\sum_{c\in A} Y^cT^c$ by setting, for each $c\neq b$, $Y^c \eqlaw X^c$ and independent of $X$, and setting $Y^b = \CR X^a$.
\end{definition}
Note that $Y$ can be defined on a larger probability space than that of $X$ and that $Y\eqlaw X$ by construction. The following is the main result of this article.
\begin{theorem}\label{thm:main}
Suppose $B$ satisfies Assumption~\ref{assump:not-deriv} and $\{X^N\}_{N\geq 1}$ satisfies Assumption~\ref{assump:GFS}.
Define $Y^N$ as in Definition~\ref{def:Y}, $u_0 \eqdef X^N + Y^N$, and let $u$ be the solution to~\eqref{eq:PDE}.
Then there exist $M,c>0$ such that, for every $\eps>0$, if we set $T = (\log N)^{-M}$ for $N\geq 2$,
\begin{equ}\label{eq:prob_norm_inflation}
\lim_{N\to\infty}\P[u \textnormal{ exists on } [0,T]\times \T^d\;\;\&\;\; |\Pi_0 u_T| > c\log\log\log N] = 1\;.
\end{equ}
\end{theorem}

The proof of Theorem~\ref{thm:main} is given in Section~\ref{sec:proof_thm}.
Before continuing, we note that Theorem~\ref{thm:main} clearly proves Theorem~\ref{thm:main_informal}.
We can now also give the proof of Corollary~\ref{cor:norm_inflation}.

\begin{proof}[Proof of Corollary~\ref{cor:norm_inflation}]
Let $X$ be an $E$-valued GFS with $\E|X^a_k|^2 = \sigma^2(k)$ for every $a\in A$,
and let $Y$ be defined as in Definition~\ref{def:Y}.
It follows from the assumption that $\init$ carries the GFS with variances $\sigma^2(k)$ that the law of $Z\eqdef X+Y$ is a Gaussian measure on the Banach space $\hat\init$
and that
\begin{equ}\label{eq:Z_conv_prob}
\lim_{N\to\infty}\|Z^N-Z\| = 0\quad  \text{in probability.}
\end{equ}
By the assumption that the law of $X^a$ has full support in $\init$, it follows that $Z$ has full support in $\hat\init$.
(One only needs to be careful about the component $T^b$ for which $Z^b = X^b + \CR X^a$ since this is not independent of $Z^a = X^a+Y^a$. However, since $X^b$ is independent of $\{X^c\}_{c\neq b}$ and of $Y$, we indeed obtain that the law of $Z$ has full support in $\hat\init$.)

Now, for every $x\in\hat\init$ and $\delta>0$, one has $\eps\eqdef \P[\|Z-x\| < \delta/2]>0$.
Furthermore,
it follows from~\eqref{eq:Z_conv_prob} that, for all $N$ sufficiently large,
\begin{equ}
\P[\|Z^N-x\| < \delta]>\eps/2\;.
\end{equ}
The conclusion follows from Theorem~\ref{thm:main} (or, more simply, Theorem~\ref{thm:main_informal}).
\end{proof}

\begin{remark}\label{rem:full_support}
In the context of Corollary~\ref{cor:norm_inflation},
a sufficient condition for the law of the real GFS $X$ in~\eqref{eq:GFS_conv} to have full support in $\init$ is that $\sigma^2(k)>0$ for all $k\in\Z^d$ and that the smooth functions are dense in $\init$.
Indeed, the condition $\sigma^2(k)>0$ implies that the Cameron--Martin space of $X$ contains all trigonometric functions, which are dense in the smooth functions.
\end{remark}

We briefly outline the proof of Theorem~\ref{thm:main} and the structure of the rest of the article.
Dropping the reference to $N$,
in Section~\ref{sec:decor} we show that all the quadratic terms in $B(\CP_t(X+Y),D\CP_t(X+Y))$ are controlled uniformly in $N\geq 1$ and $t\in(0,1)$ \textit{except} for the spatial mean
\begin{equ}
J \eqdef \Pi_0 [B_1(\CP_t X^aT^a,\partial_1\CP_tY^bT^b)+B_1(\CP_t Y^bT^b,\partial_1 \CP_t X^aT^a))]\;.
\end{equ}
In Section~\ref{sec:cor} we show that $J - \E J$ is controlled uniformly in $N,t$, together with an upper bound $\E J\leq C N^2\wedge t^{-1}$ uniformly in $N,t$ and a lower bound as $t\to0$ and $N>t^{-1/2}$
\begin{equ}\label{eq:heuristic}
t^{-1}|\log t|^{-1}(\log|\log t|)^{-1}\leq C \E J\;.
\end{equ}
Finally, in Section~\ref{sec:determin},
we show that the solution $u_t$ to~\eqref{eq:PDE}
tracks closely the first Picard iterate $\CP_t u_0 + \int_0^t B(\CP_s u_0,D\CP_s u_0)\mrd s$ at all times $t \leq (\log N)^{-M}$.
Since this first Picard iterate blows up like $\log \log \log N$ at time $t= (\log N)^{-M}$ due to~\eqref{eq:heuristic},
this allows us to conclude the proof of Theorem~\ref{thm:main} in Section~\ref{sec:proof_thm}.
In Appendix~\ref{app:GFS}, we give conditions under which the Besov space $B^{\alpha}_{\infty,q}$ for $\alpha\in\R$ and $q\in [1,\infty]$
carries the GFS with a specified covariance $\sigma^2$.

\section{Decorrelated terms}\label{sec:decor}

In the remainder of the article, we will use $x\lesssim y$ to denote that $x\leq Cy$ for some fixed proportionality constant $C>0$,
and likewise let
$x\asymp y$ to denote $C^{-1}x \leq y \leq Cx$.
We let $x \gg 1$ and $x\ll 1$ denote $x>C$ and $0<x<C^{-1}$ respectively for some sufficiently large fixed $C>0$.

In this section, we let $X$ and $Y$ be real GFSs such that $\E X_k Y_n = 0$ for all $k\neq -n$ and such that there exists $C>0$ such that, for all $k\in\Z^d$
\begin{equ}
\E|X_k|^2 + \E|Y_k|^2 \leq  C (|k|+1)^{-d+1}\;
\end{equ}
We will apply the results of this section to
the three cases $X=Y$, $\CR X=Y$, or $X$ and $Y$ are independent.
We record the following lemma, the proof of which is straightforward and we omit.

\begin{lemma}\label{lem:correlations}
Consider $m,n,\bar m,\bar n\in \Z^d$ with $m+n\neq 0$. Then
$\E X_mY_n X_{\bar m}Y_{\bar n} = 0$
unless $(-m,-n)=(\bar m,\bar n)$ or $(-m,-n)=(\bar n,\bar m)$.
\end{lemma}
Let $\pi_0 \xi = \xi - \Pi_0 \xi$ denote the projection onto the zero-mean distributions.
\begin{lemma}\label{lem:non_zero_modes}
Consider $1\leq i\leq d$, $\delta\in (0,1]$ and $\beta <0$ such that $\beta+2(1-\delta) < 0$ and $\delta>1-\frac d4$.
Then every moment of $\sup_{t\in (0,1)} t^{\delta}|\pi_0(\CP_t X \partial_i \CP_t Y)|_{\CC^\beta}$ is finite and depends only on $C,\delta,\beta$.
Furthermore, if $X$ and $Y$ are independent, then the same holds for $\sup_{t\in (0,1)} t^{\delta}|(\CP_t X \partial_i \CP_t Y)|_{\CC^\beta}$.
\end{lemma}

\begin{proof}
Let us denote $Z_t = t^\delta\CP_t X \partial_i \CP_t Y$ for $t\in (0,1]$ and
$Z_0=0$.
We will show that there exists $\kappa>0$ such that for all $p\geq 1$
\begin{equ}\label{eq:moment_bound}
\E|\pi_0(Z_u-Z_t)|^p_{\CC^\beta} \lesssim |u-t|^{\kappa p}
\end{equ}
uniformly over $0\leq t < u \leq 1$,
from which the first statement follows by Kolmogorov's continuity theorem~\cite[Thm.~A.10]{FV10}.
The second statement will be shown in a similar way.
It suffices to consider $t<u\leq 2t$.

For every $p\geq 1$ and $\ell\geq 0$,
\begin{equ}\label{eq:E1}
\E |\Delta_\ell \pi_0(Z_u-Z_t)|_{L^{2p}}^{2p}
= \int_{\T^d} \E |(\Delta_\ell \pi_0(Z_u-Z_t))(x)|^{2p} \mrd x\;.
\end{equ}
By equivalence of Gaussian moments,
the final integral is bounded above by a constant $C(p)>0$
times
\begin{equ}\label{eq:E2}
\int_{\T^d} (\E |(\Delta_\ell \pi_0(Z_u-Z_t))(x)|^2)^p \mrd x = \int_{\T^d} \Big|\E \Big(\sum_{k\neq 0} \chi_\ell(k)\F^k_{t,u} e_k(x)\Big)^2\Big|^p \mrd x\;,
\end{equ}
where $\F^k_{t,u}\eqdef\scal{Z_u-Z_t,e_k}$ is the $k$-th Fourier coefficient of $Z_u-Z_t$
\begin{equ}
\F^k_{t,u} =
\sum_{m+n=k} \psi_{t,u}(m^2+n^2) \mbi n_i
X_m Y_n\;,
\end{equ}
where
\begin{equ}
\psi_{t,u}(x) = u^\delta e^{- x u} - t^\delta e^{- x t}\;.
\end{equ}
For $0\neq k\neq -\bar k$, observe that $\E \F^k_{t,u} \F^{\bar k}_{t,u} = 0$
due to Lemma~\ref{lem:correlations}.
Hence
\begin{equ}
\E \Big(\sum_{k\neq 0} \chi_\ell(k)\F^k_{t,u} e_k(x)\Big)^2 = \sum_{k\neq 0}\chi_\ell(k)^2 \E |\F^k_{t,u}|^2\;.
\end{equ}
Again by Lemma~\ref{lem:correlations}, for $k\neq 0$,
\begin{equs}
\E |\F^k_{t,u}|^2 &= \sum_{m+n=k} \psi_{t,u}(m^2+n^2)^2
(n_i^2\E X_m Y_n X_{-m}Y_{-n} + n_im_i\E X_mY_nX_{-n}Y_{-m})
\\
&\lesssim \sum_{m+n=k} \psi_{t,u}(m^2+n^2)^2 
n_i^2 (|m|+1)^{-d+1}(|n|+1)^{-d+1}\;.\label{eq:sum_F}
\end{equs}
Since $0<\delta\leq 1$, observe that, for $\kappa>0$ small, uniformly in $x\geq 1$ and $0<t<u\leq2t<1$,
\begin{equ}\label{eq:psi_bound}
\psi_{t,u}(x)^2 \lesssim
t^{2\delta} e^{-2x t}|t-u|^\kappa x^\kappa + e^{-2xu}|t-u|^{2\delta}
\lesssim
e^{-2xt}t^{2\delta-\kappa}|t-u|^\kappa x^\kappa
\lesssim |t-u|^\kappa x^{-2\delta+2\kappa}\;,
\end{equ}
where we used
that $e^{-x}x^p\lesssim 1$ for any $p>0$ in the final bound.

We split the sum~\eqref{eq:sum_F} into two parts with $|m| \leq 2|k|$ and $|m|>2|k|$.
The first part is bounded above by a multiple of
\begin{equ}
|t-u|^\kappa k^{-4\delta+4\kappa}
\sum_{1 \leq |m|\leq 2|k|}
m^{-d+1}|k-m|^{-d+3}
\lesssim
|t-u|^\kappa k^{-d+4+4\kappa-4\delta}\;,
\end{equ}
where we used~\eqref{eq:psi_bound} and
\begin{equ}
\sum_{1\leq |m|\leq|k|/2} m^{-d+1}|k-m|^{-d+3}
\asymp 
\sum_{|k|/2\leq |m|\leq2|k|} m^{-d+1}|k-m|^{-d+3}
\asymp k^{-d+4}\;.
\end{equ}
The second part is bounded by a multiple of
\begin{equ}
|t-u|^\kappa
\sum_{|m|> 2|k|}
m^{-4\delta+4\kappa} m^{-2d+4}
\asymp |t-u|^\kappa k^{-4\delta+4\kappa+4-d}\;,
\end{equ}
where we again used~\eqref{eq:psi_bound} and the fact that $-4\delta+4\kappa-2d+4<-d$ for $\kappa>0$ small since $\delta>1-\frac d4$ by assumption.
In conclusion, $\E |\F^k_{t,u}|^2 \lesssim |t-u|^\kappa k^{-4\delta+4\kappa+4-d}$,
and thus
\begin{equ}\label{eq:E3}
\E \Big(\sum_{k\neq 0} \chi_\ell(k)\F^k_{t,u} e_k(x)\Big)^2 \lesssim
|t-u|^\kappa 2^{\ell (-4\delta +4\kappa+4)}\;.
\end{equ}
Consequently, by~\eqref{eq:E1},~\eqref{eq:E2}, and~\eqref{eq:E3},
\begin{equs}
\E |\pi_0(Z_u-Z_t)|_{B^{\beta}_{2p,2p}}^{2p} &= \sum_{\ell\geq -1}2^{2p\ell\beta} \E|\Delta_\ell \pi_0(Z_u-Z_t) |_{L^{2p}}^{2p}
\lesssim
\sum_{\ell\geq -1}
|t-u|^{\kappa p}
2^{2p\ell\beta+\ell(4-4\delta +4\kappa)p}\;.
\end{equs}
Since $2\beta+4-4\delta<0$, we can take $\kappa>0$ sufficiently small such that the series is summable in $\ell$.
Recall the Besov embedding $B^{\alpha}_{p_1,q_1} \hookrightarrow B^{\alpha-d(1/p_1-1/p_2)}_{p_2,q_2}$
for $1\leq p_1\leq p_2\leq \infty$, $1\leq q_1\leq q_2\leq \infty$, and $\alpha\in \R$ (see, e.g.,~\cite[Lem.~A.2]{GIP15}).
Taking $p_2=q_2=\infty$, $\alpha=\beta$, and $p_1=q_1=2p$,
we obtain~\eqref{eq:moment_bound} (with $\beta$ replaced by $\beta-d/(2p)$ but where we can take $p$ arbitrarily large),
which completes the proof of the first statement.

For the second statement, since $\E\F^0_{t,u}\F^k_{t,u}=0$ for $k\neq 0$, it suffices to show
$\E|\F^0_{t,u}|^2 \lesssim |t-u|^\kappa$ uniformly in $0<t<u\leq 2t<1$.
If $X$ and $Y$ are independent, then clearly
$\E X_mY_n X_{\bar m}Y_{\bar n} = \E X_mX_{\bar m}\E Y_nY_{\bar n}$, and thus, for $\kappa>0$ small,
\begin{equ}
\E |\F^0_{t,u}|^2 = \sum_{m\in\Z^d} \psi_{t,u}(2m^2)^2
(m_i^2\E |X_m|^2 \E |Y_m|^2)
\lesssim \sum_{m\neq 0} \psi_{t,u}(2m^2)^2 
m_i^2 m^{-2d+2} \lesssim |t-u|^\kappa\;,
\end{equ}
where the final bound follows from~\eqref{eq:psi_bound} and the fact that $-4\delta-2d+4<-d$.
\end{proof}

\section{Correlated terms}\label{sec:cor}

We now take  $\{X^N\}_{N \geq 1}$ satisfying Assumption~\ref{assump:GFS}.
We will drop $N$ from our notation, writing simply $X$ for $X^N$.
Recall that Assumption~\ref{assump:not-deriv} is in place.
We define $Z\colon [0,1]\to \R$ by
\begin{equ}
Z_t = \sum_{n_1>0} 2 e^{-2n^2t} n_1|X^a_n|^2\;.
\end{equ}
The importance of $Z_t$ is that, if $\xi\in C^\infty(\T^d,\R)$ and $1\leq i\leq d$,
then
\begin{equs}
\Pi_0 (\xi \partial_i\CR \xi) &=
\sum_n \mbi n_i\xi_{-n}(\CR\xi)_{n}
=
\sum_{n_1>0} 
(-n_i)
|\xi_n|^2
+ \sum_{n_1<0}
n_i
|\xi_n|^2
+ \sum_{n_1=0}
\mbi n_i
|\xi_n|^2
\\
&=
\begin{cases}
-\sum_{n_1>0} 2 n_1|\xi_n|^2 &\quad\textnormal{ if } i=1\;,\\
0 &\quad\textnormal{ if } i\neq 1\;.
\end{cases}
\end{equs}
(If $i\neq 1$, the 1st and 2nd sums cancels and the 3rd is $0$, and if $i=1$, the 3rd sum is still $0$ but the 1st and 2nd sums are equal.)
Similarly
\begin{equs}
\Pi_0 (\CR\xi \partial_i \xi) &=
\sum_n \mbi n_i (\CR\xi)_{-n} \xi_{n}
=
\sum_{n_1>0} 
n_i
|\xi_n|^2
+ \sum_{n_1<0}
(-n_i)
|\xi_n|^2
+ \sum_{n_1=0}
\mbi n_i
|\xi_n|^2
\\
&=
\begin{cases}
\sum_{n_1>0} 2 n_1|\xi_n|^2 &\quad\textnormal{ if } i=1\;,\\
0 &\quad\textnormal{ if } i\neq 1\;.
\end{cases}
\end{equs}
Therefore
\begin{equ}
\Pi_0(\CP_t \CR X^a \partial_i \CP_t X^a)
=
- \Pi_0(\CP_t X^a \partial_i \CP_t \CR X^a)
=
\begin{cases}
Z_t &\quad\textnormal{ if } i=1\;,\\
0 &\quad\textnormal{ if } i\neq 1\;.
\end{cases}
\end{equ}
Hence, for $X,Y$ related as in Definition~\ref{def:Y}, by Lemma~\ref{lem:non_zero_modes}, the quantity which we expect to be well-behaved uniformly in $N$ is
\begin{equ}
B(\CP_t (X+Y), D\CP_t(X+Y)) + (B_1(T^a,T^b)-B_1(T^b,T^a))Z_t\;.
\end{equ}
The following lemma shows that $Z_t$ is well approximated by its expectation.

\begin{lemma}\label{lem:Z-EZ}
Let $\delta> 1 -\frac d4$.
Then
$\sup_{t\in (0,1)} t^{\delta}( Z_t-\E Z_t)$ has moments of all orders bounded uniformly in $N$.
\end{lemma}

\begin{proof}
Define $W_t =t^{\delta} ( Z_t-\E Z_t)$.
We will show that there exists $\kappa>0$ such that, for all $p\geq 1$,
\begin{equ}
\E |W_u-W_t|^p \lesssim |t-u|^{\kappa p}\;,
\end{equ}
uniformly in $0\leq t < u \leq 1$ and $N\geq 1$,
from which the conclusion follows by Kolmogorov's continuity theorem~\cite[Thm.~A.10]{FV10}.
It suffices to consider $t<u\leq 2t$ and, by equivalence of Gaussian moments, $p=2$.

Since $X^a_n$ and $X^a_m$ are independent if $n\neq m$ with $n_1,m_1>0$, and $\E |X^a_n|^4 - (\E |X_n^a|^2)^2 \leq \sigma^4(n) \lesssim n^{-2d+2}$,
we obtain 
\begin{equ}
\E(W_u - W_t)^2 \asymp \sum_{n_1>0} \psi_{t,u}(n^2)^2 n_1^2 ( \E |X^a_n|^4 - (\E |X_n^a|^2)^2)
\lesssim \sum_{n_1>0} \psi_{t,u}(n^2)^2  n^{-2d+4}\;,
\end{equ}
where 
\begin{equ}
\psi_{t,u}(x) = t^{\delta}e^{-2xt}-u^{\delta}e^{-2xu}\;.
\end{equ}
Recall from~\eqref{eq:psi_bound} that, for $\kappa>0$ small, uniformly in $0<t<u\leq 2t<1$ and $x\geq 1$,
\begin{equ}
\psi_{t,u}(x)^2 \lesssim |t-u|^\kappa x^{-2\delta+2\kappa}\;.
\end{equ}
We can take $\kappa>0$ sufficiently small such that $-2d+4-4\delta+4\kappa < -d$,
and therefore
\begin{equation*}
\E(W_u - W_t)^2
\lesssim |t-u|^\kappa \sum_{n_1>0} n^{-2d+4-4\delta+4\kappa}
\asymp |t-u|^\kappa\;.\qedhere
\end{equation*}
\end{proof}
We now show that $\E Z_t$ explodes as $N\to\infty$ and $t\to 0$ in a controlled manner.
\begin{lemma}\label{lem:E_Z}
It holds that
\begin{equ}\label{eq:_Z_upper}
\E Z_t \lesssim N^2 \wedge t^{-1}
\end{equ}
uniformly in $t\in (0,1)$ and in $N\geq 1$.
Furthermore,
\begin{equ}\label{eq:_Z_lower}
\E Z_t \gtrsim t^{-1}|\log t|^{-1} (\log |\log t| )^{-1}
\end{equ}
uniformly in $0<t\ll 1$ and $N>t^{-1/2}$.
\end{lemma}

\begin{proof}
Using the upper bound in Assumption~\ref{assump:GFS},
$\E Z_t$ is bounded above by a multiple of
\begin{equ}
\sum_{1\leq |n|\leq N} e^{-n^2 t} n^{-d+2}
\lesssim
\int_1^{2N}  e^{-r^2 t} r\mrd r
=
\int_{t^{1/2}}^{2N t^{1/2}}  e^{- u^2} u t^{-1/2} t^{-1/2}\mrd u
\asymp t^{-1} \wedge N^2\;.
\end{equ}
On the other hand,
using the lower bound in Assumption~\ref{assump:GFS},
$\E Z_t$ is bounded below by $cK$ where $c>0$ is uniform in $t\ll 1$ and $N>t^{-1/2}$, and
\begin{align*}
K
&=\sum_{3\leq |n|\leq N} e^{-n^2 t} n^{-d+2}|\log n|^{-1}|\log \log n|^{-1}
\\
&\gtrsim
\int_4^{N/2}  e^{-r^2 t} r|\log r|^{-1}|\log \log r|^{-1}\mrd r
\\
&=
\int_{4t^{1/2}}^{N t^{1/2}/2}  e^{- u^2} u t^{-1/2} |\log (ut^{-1/2})|^{-1}|\log \log (ut^{-1/2})|^{-1}t^{-1/2}\mrd u
\\
&\gtrsim t^{-1}|\log t|^{-1} (\log |\log t| )^{-1}\;.\qedhere
\end{align*}
\end{proof}
Define the deterministic processes $H,I \colon [0,1] \to E$ by
\begin{equ}\label{eq:H_I_def}
H_t \eqdef (B_1(T^a,T^b)-B_1(T^b,T^a))\E Z_t\;,\qquad I_t \eqdef \int_0^t H_s\mrd s\;.
\end{equ}
Observe that, by~\eqref{eq:_Z_upper}, uniformly in $t\in (0,1)$ and $N\geq1$,
\begin{equ}\label{eq:I_bound_upper}
|I_t| = \int_0^t |H_s|\mrd s
\lesssim \int_0^{t\wedge N^{-2}} N^2 \mrd s + \int_{t\wedge N^{-2}}^t s^{-1} \mrd s
= (N^2 t)\wedge 1 + \log((N^{2}t)\vee 1) \;,
\end{equ}
while by~\eqref{eq:_Z_lower}, uniformly in $0<t\ll1$ and $N>t^{-1/2}$,
\begin{equ}\label{eq:I_bound_lower}
|I_t|
\gtrsim \int_{N^{-2}}^t s^{-1}|\log s|^{-1} (\log |\log s| )^{-1} \mrd s
= \log\log\log N^{2} -\log\log\log (t^{-1})\;.
\end{equ}

\begin{lemma}\label{lem:int_I_bound}
Consider $a,b>-1$ and $p\geq 1$. Then, uniformly in $N\geq 2$ and $t\in (0,1)$,
\begin{equ}
\int_0^t (t-s)^a|I_s|^p s^b \mrd s
\lesssim  t^{a+b+1}(\log N)^p\;.
\end{equ}
\end{lemma}

\begin{proof}
If $t\leq 2N^{-2}$, then $|I_t|\lesssim  N^2t$ by \eqref{eq:I_bound_upper},
and therefore
\begin{equation*}
\int_0^t (t-s)^a|I_s|^p s^b \mrd s 
\lesssim \int_0^{t} (t-s)^a N^{2p} s^{b+p} \mrd s
\asymp
t^{a+b+p+1} N^{2p} \lesssim t^{a+b+1}\;.
\end{equation*}
On the other hand, if $t>2N^{-2}$ then $(t-s)^a\asymp t^a$ for $s\leq N^{-2}$ and therefore, by~\eqref{eq:I_bound_upper},
\begin{align*}
\int_0^t (t-s)^a|I_s|^p s^b \mrd s 
&\lesssim \int_0^{N^{-2}} (t-s)^a N^{2p} s^{b+p} \mrd s
+\int_{N^{-2}}^t (t-s)^a  s^b (\log N )^p \mrd s
\\
&\asymp t^{a}(N^{-2})^{b+1}
+ t^{a+b+1}(\log N)^p
\lesssim t^{a+b+1}(\log N)^p\;.\qedhere
\end{align*}
\end{proof}

\section{Deterministic bounds}\label{sec:determin}

In this section, we let $H,I\colon [0,1]\to E$ be defined as in~\eqref{eq:H_I_def} from some $\{X^N\}_{N\geq 1}$ satisfying Assumption~\ref{assump:GFS}.
We consider $N\gg1$ and suppress it from our notation.

We further fix $\delta \in (\frac12,1)$ and $\beta\in (-1,0)$ such that $\hat\beta\eqdef \beta+2(1-\delta)\in (-\frac12,0)$.
We also fix $\eta \in (-\frac23,-\frac12)$ such that $\eta + \hat\beta>-1$.
For $\CQ \colon (0,T)\to \CC^\beta$,
define the norm
\begin{equ}
|\CQ|_{\CC^\beta_{\delta;T}} \eqdef \sup_{t\in (0,T)} t^{\delta}|\CQ|_{\CC^\beta}\;.
\end{equ}

\begin{lemma}\label{lem:deterministic}
Suppose $u_0\in C^\infty$ and define the functions $\CN_t = B(\CP_t u_0,D\CP_t u_0)$ and $\CQ_t = \CN_t - H_t$.
Then there exist $\kappa>0$ small and $C>0$ large, not depending on $N$, such that for all $T\in(0,1)$ with
\begin{equ}\label{eq:small_time}
T^\kappa \leq \frac{1}{C(\log N)^3(1+|\CQ|_{\CC^\beta_{\delta;T}}+|u_0|_{\CC^\eta})^{3}}\;,
\end{equ}
it holds that the solution $u$ to~\eqref{eq:PDE} exists on $[0,T]$ and
\begin{equ}
\sup_{t\in [0,T]}|u_t -\CP_t u_0 - I_t|_{\CC^{\hat\beta}} \leq C(1+|\CQ|_{\CC^\beta_{\delta;T}})\;.
\end{equ}
\end{lemma}

\begin{proof}
We decompose
$
u_t =
\CP_t u_0 + R_t + I_t
$
where $R$ solves the fixed point
\begin{equs}
R_t &= \CM_t(R) \eqdef \int_0^t \CP_{t-s}(B(u_s, Du_s) + P(u_s)) \mrd s - I_t\label{eq:CM_def}
\\
&= \int_0^t \CP_{t-s} [\CQ_s + B(u_s, D R_s ) + B( R_s + I_s, D\CP_s u_0) + P(u_s)]\mrd s\;,
\end{equs}
where we used that $H_t$ and $I_t$ are constant in space in the final equality.
Define the Banach space $\CX = \CX_T = \{f \in C((0,T), C^1)\,:\, |f|_\CX<\infty\}$ with norm
\begin{equ}
|f|_{\CX} \eqdef \sup_{t\in (0,T)} |f_t|_{\CC^{\hat\beta}} +
t^{-\hat\beta/2}|f_t|_{L^\infty} + t^{-\frac{\hat\beta}2+\frac12} |f_t|_{C^1}\;.
\end{equ}
We will show that, for $T$ satisfying~\eqref{eq:small_time}, $\CM$ is a contraction on the ball of radius $K\asymp 1+|\CQ|_{\CC^\beta_{\delta;T}}$ in $\CX$, from which the conclusion follows.

Define $\CY_t = \int_0^t \CP_{t-s} \CQ_s\mrd s$. Recalling the heat flow estimates for $\alpha<0$ and $\gamma\geq 0$ (see, e.g.~\cite[Lem.~A.7]{GIP15} and the proof of~\cite[Thm.~2.34]{BookChemin})
\begin{equ}[eq:heat_flow]
t^{-\alpha/2}|\CP_t f|_{L^\infty} + t^{1/2-\alpha/2}|\CP_t f|_{C^1} + t^{\gamma/2}|\CP_t f|_{\CC^{\alpha+\gamma}}\lesssim  |f|_{\CC^\alpha}\;,
\end{equ}
we obtain
\begin{equ}
|\CY_t|_{\CC^{\hat\beta}} \lesssim |\CQ|_{\CC^\beta_{\delta;t}}\int_0^t (t-s)^{\beta/2-\hat\beta/2} s^{-\delta}\mrd s \asymp |\CQ|_{\CC^\beta_{\delta;t}}\;.
\end{equ}
Similarly
$
|\CY_t|_{L^\infty}
\lesssim 
|\CQ|_{\CC^\beta_{\delta;t}}t^{\hat\beta/2}
$
and $|\CY_t|_{C^1} \lesssim |\CQ|_{\CC^\beta_{\delta;t}}t^{\hat\beta/2-\frac12}$ (in the last bound we used $\beta>-1$ so that $\beta/2 - \frac12 >-1$),
thus $
|\CY|_{\CX} \lesssim |\CQ|_{\CC^\beta_{\delta;T}}$.
Furthermore,
\begin{equs}
t^{-\hat\beta/2}\int_0^t |\CP_{t-s} B(\CP_s u_0, DR_s) + B(R_s,D\CP_s u_0)|_{L^\infty}\mrd s
&\lesssim |R|_{\CX}|u_0|_{\CC^\eta} t^{-\hat\beta/2}\int_0^t s^{\eta/2+\hat\beta/2-\frac12} \mrd s
\\
&\asymp t^{\frac12+\frac\eta2}|R|_{\CX}|u_0|_{\CC^\eta}\;,
\end{equs}
where we used $\hat\beta+\eta>-1$,
and
\begin{equ}
t^{-\hat\beta/2}\int_0^t |\CP_{t-s} B(R_s, DR_s)|_{L^\infty}\mrd s
\lesssim |R|_{\CX}^2 t^{-\hat\beta/2}\int_0^t s^{\hat\beta-\frac12} \mrd s
\asymp |R|_{\CX}^2t^{\frac12+\hat\beta/2}\;,
\end{equ}
where we used $\hat\beta>-\frac12$.
Moreover, by Lemma~\ref{lem:int_I_bound} with $p=1$,
\begin{equ}
t^{-\hat\beta/2}\int_0^t |\CP_{t-s} B(I_s, D\CP_su_0)|_{L^\infty}\mrd s
\lesssim |u_0|_{\CC^\eta} t^{-\hat\beta/2}\int_0^t |I_s|s^{\eta/2-1/2} \mrd s
\lesssim |u_0|_{\CC^\eta}  t^{\frac12+\eta/2-\hat\beta/2} \log N
\end{equ}
and
\begin{equ}
t^{-\hat\beta/2}\int_0^t |\CP_{t-s} B(I_s, DR_s)|_{L^\infty}\mrd s
\lesssim |R|_{\CX}t^{-\hat\beta/2}\int_0^t |I_s|s^{\hat\beta/2-1/2} \mrd s
\lesssim t^{\frac12}|R|_{\CX} \log N\;.
\end{equ}
Exactly the same bounds hold for the norm $|\cdot|_{C^1}$ with $t^{-\hat\beta/2}$ replaced by $t^{-\hat\beta/2+\frac12}$ in the first terms.
Finally, the same bounds obviously hold for the norm $|\cdot|_{\CC^{\hat\beta}} \lesssim |\cdot|_{L^\infty}$.
This bounds the first set of terms in the definition~\eqref{eq:CM_def} of $\CM_t$.
It remains to bound $\int_0^t \CP_{t-s} P(u_s)\mrd s$ for which we have, by Lemma~\ref{lem:int_I_bound} with $p=3$,
\begin{equs}
t^{-\hat\beta/2}\int_0^t |\CP_{t-s} P(u_s)|_{L^\infty}\mrd s
&\lesssim t^{-\hat\beta/2}\int_0^t(1+ s^{\eta/2}|u_0|_{\CC^\eta} + s^{\hat\beta/2}|R|_{\CX}
+ |I_s|)^3 \mrd s
\\
&\lesssim t^{-\hat\beta/2}(t + t^{3\eta/2+1}|u_0|_{\CC^\eta}^3 + t^{3\hat\beta/2+1}|R|^3_{\CX} + t(\log N)^3)\;.
\end{equs}
Again the same bound obviously holds for the norm $|\cdot|_{\CC^{\hat\beta}} \lesssim |\cdot|_{L^\infty}$,
and also for the norm $|\cdot|_{C^1}$ with $t^{-\hat\beta/2}$ replaced by $t^{-\hat\beta/2+\frac12}$ in the first term.
In conclusion, for $\kappa=\min\{\frac12+\frac\eta2+\frac{\hat\beta}2,\frac12+{\hat\beta},\frac{3\eta}{2}+1\}$,
\begin{equ}
|\CM(R)|_{\CX} \lesssim |\CQ|_{\CC^\beta_{\delta;T}}
+ T^\kappa(\log N)^3 (1+|R|_{\CX} +|u_0|_{\CC^\eta})^3\;. 
\end{equ}
In exactly the same way,
\begin{equ}
|\CM(R) - \CM(\bar R)|_{\CX} \lesssim T^{\kappa}(\log N)^3|R-\bar R|_\CX(1+|u_0|_{\CC^\eta} + |R|_{\CX}+|\bar R|_{\CX})^2\;.
\end{equ}
It follows that for $K\asymp 1+|\CQ|_{\CC^\beta_{\delta;T}}$
and $T^\kappa(\log N)^3(K + |u_0|_{\CC^\eta})^3\lesssim K$, $\CM$ is a contraction on the ball in $\CX$ of radius $K$ as desired.
\end{proof}

\section{Proof of Theorem~\ref{thm:main}}\label{sec:proof_thm}

We are now ready for the proof of Theorem~\ref{thm:main}.
Consider $M>3/\kappa$ for the $\kappa$ from Lemma~\ref{lem:deterministic}, and define $T= (\log N)^{-M}$ for $N\gg1$.
Define $\CN,\CQ\colon [0,1]\to C^\infty$ as in Lemma~\ref{lem:deterministic}
and $I\colon[0,1]\to E$ as in~\eqref{eq:H_I_def}.
Consider $\delta,\beta,\eta,\hat\beta$ as in the previous section
such that furthermore $\delta>1-\frac d4$.

The moments of $|u_0|_{\CC^\eta}$ are bounded uniformly in $N$ (Lemma~\ref{lem:Besov_moments}), and, by Lemmas~\ref{lem:non_zero_modes} and~\ref{lem:Z-EZ}, so are
the moments of $|\CQ|_{\CC^\beta_{\delta;1}}$.
Therefore, for $C$ as in Lemma~\ref{lem:deterministic},
\[
\lim_{N\to\infty} \P\big [T^\kappa \leq C^{-1}(\log N)^{-3}(1+|\CQ|_{\CC^\beta_{\delta;T}}+|u_0|_{\CC^\eta})^{-3}
\big] = 1\;.
\]
Hence, by Lemma~\ref{lem:deterministic},
\begin{equ}
\lim_{N\to\infty}\P\big[u \textnormal{ exists on } [0,T]\times\T^d \;\; \&\;\;
|u_T - \CP_T u_0 - I_T|_{\CC^{\hat\beta}} \leq C(1+|\CQ|_{\CC^\beta_{\delta;1}})\big] = 1\;.
\end{equ}
Finally, since $|I_T| \gtrsim \log\log\log N$ by~\eqref{eq:I_bound_lower}, and
\begin{equ}
|\Pi_0(u_T - \CP_Tu_0 - I_T)| + |\Pi_0\CP_Tu_0| < |I_T|/2 \Rightarrow
|\Pi_0 u_T| > |I_T|/2 \;,
\end{equ}
we obtain~\eqref{eq:prob_norm_inflation} for $c>0$ sufficiently small, which concludes the proof of Theorem~\ref{thm:main}. \qed

\appendix

\section{Besov regularity of Gaussian free fields}\label{app:GFS}

In this appendix, we consider
a (real) GFS $X$ (see Definition~\ref{def:GFS}) 
and denote $\sigma^2(n)\eqdef \E|X_n|^2$.
If $\sigma^2(n)\lesssim |n|^{\gamma}$ for $n\neq 0$, then, for any $\eta<-(d+\gamma)/2$,
it is classical that $X \in \CC^{\eta}$ and $|X^N-X|_{\CC^\eta} \to 0$ a.s. as $N\to\infty$, where $X^N\eqdef \Pi_N X$.
The following result shows that, if $\sigma^2(n)$ is logarithmically smaller,
then we have improved behaviour.

\begin{proposition}\label{prop:Besov_reg}
Let $\gamma,\theta,\eta\in \R$ and suppose
\begin{equ}
K\eqdef \sup_{|n|\geq 3}\frac{\sigma^2(n)}{n^{\gamma}(\log n)^{\theta}(\log\log n)^{\eta}} <\infty.
\end{equ}
\begin{enumerate}[label=(\roman*)]
\item\label{pt:-1} If $\theta=-1$ and $\eta<0$, then almost surely $
\lim_{N\to\infty}|X^N-X|_{\CC^{-(d+\gamma)/2}}= 0$.

\item\label{pt:<-1} If $\theta<-1$ and $q > -2/(\theta+1)$, then almost surely $
\lim_{N\to\infty}|X^N-X|_{B^{-(d+\gamma)/2}_{\infty,q}}= 0$.
Furthermore,
$\sup_N \E |X^N|_{B^{-(d+\gamma)/2}_{\infty,q}}^p$
is finite  for all $p\in [1,\infty)$ and depends only on $p,q,\gamma,\theta,\eta,d,K$, and $\{\sigma^2(n)\}_{|n|\leq 2}$.
\end{enumerate}
\end{proposition}

Part~\ref{pt:-1} of Proposition~\ref{prop:Besov_reg} is used to prove Corollary~\ref{cor:endpoint},
while part~\ref{pt:<-1} proves Lemma \ref{lem:Besov_moments}.
For the proof of Proposition~\ref{prop:Besov_reg},
recall that if $Z$ is an $\R$-valued centred continuous Gaussian process on $\T^d$ (or, more generally, on a totally bounded metric space), then Dudley's theorem (see, e.g.,~\cite[Thm.~3.18]{Massart07}) implies that
\begin{equ}\label{eq:max_expectation}
\E |Z|_{L^\infty} \leq 24 \int_0^\sigma \sqrt{\log N_\rho(r)}\mrd r\;,
\end{equ}
where $N_\rho(r)$ is the minimum number of $\rho$-balls of radius $r/2$ needed to cover $\T^d$, $\rho$ is the metric on $\T^d$ defined by
\begin{equ}
\rho(x,y)^2 = \E[|Z(x)-Z(y)|^2]\;,
\end{equ}
and
\begin{equ}
\sigma^2 = \sup_{x\in\T^d} \E |Z(x)|^2\;.
\end{equ}
Furthermore, one has the concentration inequality (see~\cite[Prop.~3.19]{Massart07}) for all $M\geq 0$
\begin{equ}\label{eq:concentration}
\P[||Z|_{L^\infty} - \E |Z|_{L^\infty}| \geq M] \leq 2e^{-\frac12 M^2/\sigma^2}\;.
\end{equ}

\begin{proof}[Proof of Proposition~\ref{prop:Besov_reg}]
Consider $\ell \geq 1$ and define
$Z_\ell = 2^{-\ell(d+\gamma)/2}\Delta_\ell (X-X^N)$.
Clearly $Z_\ell=0$ for $\ell \lesssim \log N$.
We henceforth consider $\ell\gtrsim\log N$.
Then
\begin{equ}
\sigma_{Z_\ell}^2\eqdef \E |Z_\ell(x)|^2 \lesssim \sum_{k} \chi_\ell(k)^2 2^{-\ell(d+\gamma)}\E[|X_k|^2]
\lesssim 2^{d\ell}2^{-\ell(d+\gamma)}2^{\ell \gamma}\ell^{\theta}(\log \ell)^{\eta}
=\ell^{\theta}(\log \ell)^{\eta}
\end{equ}
uniformly in $\ell,N$.
Applying~\eqref{eq:concentration}, we obtain for $M >0$,
\begin{equs}[eq:P_diff_bound]
\P[\ell^{-\theta/2}(\log \ell)^{-\eta/2}||Z_\ell|_{L^\infty} - \E |Z_\ell|_{L^\infty}| > M ]
&=
\P[||Z_\ell|_{L^\infty} - \E |Z_\ell|_{L^\infty}| > \ell^{\theta/2}(\log \ell)^{\eta/2}M]
\\
&\leq 2\exp\Big(-\frac12 \ell^{\theta}(\log \ell )^{\eta} M^2/\sigma_{Z_\ell}^2\Big)
\\
&\leq 2e^{-M^2/C}\;,
\end{equs}
for $C$ uniform in $\ell,M,N$.
To estimate $\E |Z_\ell|_{L^\infty}$, observe that
\begin{equs}
\rho(x,y)^2 &\eqdef \E|Z_\ell(x)-Z_\ell(y)|^2
\lesssim \sum_k \chi_\ell(k)^2 2^{-\ell(d+\gamma)}\E|X_k|^2|e^{\mbi k x} - e^{\mbi k y}|^2
\\
&\lesssim
2^{\ell d} 2^{-\ell(d+\gamma)} 2^{\ell  \gamma}\ell^{\theta}(\log \ell)^{\eta}(1\wedge 2^{2\ell}|x-y|^2)
\\
&= \ell^{\theta}(\log \ell)^{\eta}(1\wedge 2^{2\ell}|x-y|^2)\;.
\end{equs}
Therefore, for $r\leq \sigma_{Z_\ell}$,
\begin{equ}
B^\rho_x(r/2) \supset B_x(C 2^{-\ell}\ell^{-\theta/2}(\log\ell)^{-\eta/2}r/2)\;,
\end{equ}
where $B^\rho_x(\eps)$ and $B_x(\eps)$ are the $\rho$- and Euclidean-balls respectively centred at $x$ of radius $\eps$, and where $C>0$ does not depend on $\ell,r$.
Therefore $N_\rho(r) \lesssim 2^{\ell d}\ell^{\theta d/2}(\log\ell)^{\eta d/2}r^{-d}$.
It follows from~\eqref{eq:max_expectation} that, uniformly in $\ell$,
\begin{equ}\label{eq:E_Z_bound}
\E|Z_\ell|_{L^\infty} \lesssim \int_0^{\ell^{\theta/2}(\log\ell)^{\eta/2}} 
(\ell d - d\log r)^{1/2} \mrd r
\lesssim \ell^{\theta/2+1/2}(\log\ell)^{\eta/2}\;.
\end{equ}
\textit{Proof of~\ref{pt:-1}.} Suppose now that $\theta=-1$.
Then,
for all $\eps>0$,
\begin{equ}
\P[||Z_\ell|_{L^\infty} - \E |Z_\ell|_{L^\infty}| > \eps] \leq 2e^{-\ell(\log\ell)^{-\eta} \eps^2/C}\;,
\end{equ}
which is clearly summable in $\ell$.
Suppose further that $\eta<0$. Then 
$\E|Z_\ell|_{L^\infty} \lesssim (\log\ell)^{\eta/2} \to 0$.
Therefore,
by the Borel--Cantelli lemma, for all $\eps > 0$,
\begin{equ}
\P\Big[\limsup_{\ell\to\infty} |Z_\ell|_{L^\infty} > \eps\Big] = 0\;.
\end{equ} 
Hence, almost surely, as $N\to\infty$,
\begin{equ}
|X^N-X|_{B^{-(d+\gamma)/2}_{\infty,\infty}} \leq \sup_{\ell\gtrsim \log N} |Z_\ell|_{L^\infty} \to 0\;.
\end{equ}
\textit{Proof of~\ref{pt:<-1}.}
Suppose now $\theta<-1$, $\eta\in\R$ arbitrary, and $q > -2/(\theta+1)$.
Then
\begin{equ}[eq:XN_X_diff]
|X^N-X|_{B^{-(d+\gamma)/2}_{\infty,q}}^q
\leq \sum_{\ell\gtrsim\log N} |Z_\ell|_{L^\infty}^q
\lesssim \sum_{\ell\gtrsim\log N} ||Z_\ell|_{L^\infty}-\E |Z_\ell|_{L^\infty}|^q +
\sum_{\ell\gtrsim\log N}(\E |Z_\ell|_{L^\infty})^q\;.
\end{equ}
For $\eps>0$ sufficiently small, the second sum is bounded by a multiple of $(\log N)^{q(\theta/2+1/2)+1+\eps}$ due to~\eqref{eq:E_Z_bound} since $q(\theta/2 + 1/2)<-1$,
and thus converges to $0$ as $N\to\infty$.
Furthermore, by~\eqref{eq:P_diff_bound},
\begin{equ}
\P[||Z_\ell|_{L^\infty} - \E |Z_\ell|_{L^\infty}|^q > \ell^{q(\theta/2+1/2)}] \leq 2e^{-\ell(\log\ell)^{-\eta} /C}\;,
\end{equ}
which is summable in $\ell$.
Therefore, again by the Borel--Cantelli lemma,
\begin{equation*}
\P\Big[\lim_{N\to\infty}\sum_{\ell\gtrsim\log N} ||Z_\ell|_{L^\infty} - \E |Z_\ell|_{L^\infty}|^q = 0\Big] =1\;,
\end{equation*}
which proves the almost sure convergence.

For the final claim that $\sup_N \E |X^N|_{B^{-(d+\gamma)/2}_{\infty,q}}^p$
is finite and depends only on the given parameters, it suffices to consider $p=q$  since, by Fernique's theorem, $\E|Y|^p \leq C_p (\E|Y|)^p$ for any centred Gaussian random variable $Y$ with values in a separable Banach space\footnote{While $B^\beta_{\infty,q}$ is not separable, we can work with the closure of smooth functions therein, which is separable.} (see, e.g.  \cite[Cor.~3.2.7]{Stroock23}).
Then
\begin{equ}
\E |X^N-X|_{B^{-(d+\gamma)/2}_{\infty,q}}^q \leq \sum_{\ell\gtrsim \log N} \E |Z_\ell|_{L^\infty}^q \lesssim \sum_{\ell\gtrsim\log N} (\E |Z_\ell|_{L^\infty})^q\;,
\end{equ}
where in the final bound we again used $\E|Z_\ell|_{L^\infty}^q \lesssim  (\E|Z|_{L^\infty})^q$ by Fernique's theorem.
The final sum is finite and bounded uniformly in $N$ by the argument following \eqref{eq:XN_X_diff}.
\end{proof}

\section{Well-posedness in the regime \texorpdfstring{$\eta>-\frac12$}{eta > -1/2}}
\label{app:large_eta}

In this appendix, we prove well-posedness of \eqref{eq:PDE}
in the regime $u_0\in \CC^{\eta}$ for $\eta>-\frac12$
(we do not require here that $B$ is not a total derivative).
We follow the notation of Sections \ref{sec:intro} and \ref{subsec:Besov} (and start of Section \ref{sec:decor}).

\begin{proposition}
For all $\eta\in(-\frac12,0)$, there exists $\delta>0$ with the following property.
For $T>0$, define
the Banach space $\CX_T = \{f \in C((0,T), C^1)\,:\, |f|_{\CX_T}<\infty\}$ with norm
\begin{equ}
|f|_{\CX_T} \eqdef \sup_{t\in (0,T)} t^{-\eta-\frac12}|f_t|_{L^\infty} + t^{-\eta} |f_t|_{C^1}\;.
\end{equ}
For all $K>1$ and $u_0\in \CB^\eta_K \eqdef \{v\in \CC^\eta\,:\, |v|_{\CC^\eta}<K\}$,
if $T^{\eta+1} < \delta K^2$,
then there exists a unique $R \in \CX_T$
such that $u \eqdef \CP u_0 + R$ solves $\partial_t u = \Delta u + B(u, Du) + P(u)$ on $(0,T)\times \T^d$
and the map $\CB^\eta_K\ni u_0\mapsto R\in\mcX_T$ is $\delta^{-1} K$-Lipschitz.
\end{proposition}

\begin{proof}
For $v\in \CC^\eta$ and $R\in \CX$, let us denote $u_t = \CP_t v + R_t$ and define
\[
\CM^v_t(R) = \int_0^t \{B(u_s,Du_s) + P(u_s)\}\mrd s\;.
\]
By the elementary bound $\int_0^t s^a(t-s)^b \mrd s \asymp t^{a+b+1}$ for $a,b>-1$,
we have
\begin{align*}
t^{-\eta-\frac12}|\CM^v_t(R)|_{L^\infty}
&\lesssim t^{-\eta-\frac12}\int_0^t \{|u_s|_{L^\infty}|u_s|_{C^1}+(1+|u_s|_{L^\infty})^3\}\mrd s
\\
&\lesssim 
t^{-\eta-\frac12}
\int_0^t \{s^{\eta-\frac12}|v|_{\CC^\eta}^2 + s^{2\eta+\frac12}|R|_{\CX_T}^2 + 1 + s^{3\eta/2}|v|_{\CC^\eta}^3 + s^{3\eta+\frac32}|R|_{\CX_T}^3\} \mrd s
\\
&\asymp
|v|^2_{\CC^\eta} + t^{\eta+1}|R|_{\CX_T}^2 + t^{-\eta+\frac12} + t^{\eta/2+\frac12}|v|_{\CC^\eta}^3 + t^{2\eta+2}|R|_{\CX_T}^3\;,
\end{align*}
where we used the heat flow estimates \eqref{eq:heat_flow} in the second line and that $\eta>-\frac12$ in the third line
to estimate the integral $\int_0^t s^{\eta-\frac12}|v|_{\CC^\eta}^2\mrd s \asymp t^{\eta+\frac12}|v|^2_{\CC^\eta}$.
Using again \eqref{eq:heat_flow}, the same bound holds with $t^{-\eta-\frac12}|\CM^v_t(R)|_{L^\infty}$ replaced by $t^{-\eta}|\CM^v_t(R)|_{C^1}$.

It follows that there exists $\eps>0$ such that, for all $M>0$ sufficiently large, $K>1$, and $v\in \CB^\eta_K$,
if $T^{\eta+1}\leq \eps M^{-1}K^{-2}$,
then $\CM^v$ stabilises the ball in $\CX_T$ centred at $0$ of radius $M K^2$.

Likewise, for another $\bar v\in \CC^\eta$ and $\bar R\in \CX_T$,
\begin{equs}[eq:diff_bound]
t^{-\eta-\frac12}|\CM^v_t(R)-\CM^{\bar v}_t(\bar R)|_{L^\infty}
&\lesssim
(|v|_{\CC^\eta}+|\bar v|_{\CC^\eta}+t^{\eta/2+\frac12}(|R|_{\CX_T}+|\bar R|_{\CX_T}))
\\
&\qquad(
|v-\bar v|_{\CC^\eta}+t^{\eta/2+\frac12}|R-\bar R|_{\CX_T}
)\notag
\\
&+
(1+t^\eta(|v|_{\CC^\eta}+|\bar v|_{\CC^\eta})^2+t^{2\eta+1}(|R|_{\CX_T}+|\bar R|_{\CX_T})^2)\notag
\\
&\qquad
(
t^{-\frac\eta2+\frac12}|v-\bar v|_{\CC^\eta}+t|R-\bar R|_{\CX_T}\notag
)
\;.
\end{equs}
Again the same bound holds with $t^{-\eta-\frac12}|\cdots|_{L^\infty}$ replaced by $t^{-\eta}|\cdots|_{C^1}$.
Applying this with $v=\bar v\in \CB^\eta_K$ shows that $\CM^v$ is a contraction on the ball in $\CX_T$ centred at $0$ of radius
$MK^2$ provided $T^{\eta+1}\leq \eps M^{-1}K^{-2}$.
This proves the existence of a unique fixed point $R=\CM^v(R)$ in this ball.
The uniqueness of $R$ in all of $\CX_T$ follows by restarting the equation over sufficiently short time intervals.

Furthermore, applying \eqref{eq:diff_bound} to two fixed points $R=\CM^v(R)$ and $\bar R=\CM^{\bar v}(\bar R)$ with $v,\bar v\in \CB^\eta_K$
shows that $|R-\bar R|_{\CX_T} \lesssim \eps^{1/2}M^{1/2}K|v-\bar v|_{\CC^\eta}$ provided $T^{\eta+1}\leq \eps M^{-1}K^{-2}$,
which proves the claimed local Lipschitz property.
\end{proof}

\subsubsection*{Acknowledgements}

The author would like to thank the anonymous referees for their detailed reading and comments that have helped to improve the paper.

\subsubsection*{Data availability statement} Data sharing not applicable to this article as no datasets were generated or analysed during the current study.

\subsubsection*{Funding} The author acknowledges support by the Engineering and Physical Sciences Research Council via the New Investigator Award EP/X015688/1.

\bibliographystyle{./Martin}
\bibliography{./refs}

\end{document}